\documentclass{amsart}
\usepackage{amsmath, amssymb, amsthm, amscd, amsfonts, eucal, hyperref}
\usepackage[all]{xy}
\usepackage{enumerate, color}
\usepackage[T5, T1]{fontenc}
\usepackage{geometry, array, caption, subfig}

\newcolumntype{C}[1]{>{\centering}m{#1}}

\newtheorem{theorem}{Theorem}[section]
\newtheorem{proposition}[theorem]{Proposition}
\newtheorem{lemma}[theorem]{Lemma}
\newtheorem {corollary}[theorem]{Corollary}
\theoremstyle {definition}

\newtheorem {example}[theorem]{Example}

\theoremstyle {remark}

\def\ini{\operatorname{in}}

\def\Tor{\operatorname{Tor}}

\def\deg{\operatorname{deg}}

\def\depth{\operatorname{depth}}
\def\reg{\operatorname{reg}}

\newcommand{\bP}{\ensuremath{\mathbb P}}

\newcommand{\bZ}{\ensuremath{\mathbb Z}}

\begin{document}

\title[The reduction number and degree bound of projective subschemes]{The reduction number and degree bound\\ of projective subschemes}

\author{\fontencoding{T5}\selectfont \DJ o\`an Trung C\uhorn{}\`\ohorn ng}
\address{\fontencoding{T5}\selectfont \DJ o\`an Trung C\uhorn{}\`\ohorn ng. Institute of Mathematics and the Graduate University of Science and Technology, Vietnam Academy of Science and Technology, 18 Hoang Quoc Viet, 10307 Hanoi, Vietnam.}
\email{dtcuong@math.ac.vn}

\author{Sijong Kwak}
\address{Sijong Kwak. Department of Mathematical Sciences, Korea Advanced Institute of Science and Technology, 373-1 Gusung-dong, Yusung-Gu, Daejeon, Republic of Korea}
\email{sjkwak@kaist.ac.kr}

\thanks{\fontencoding{T5}\selectfont \DJ o\`an Trung C\uhorn{}\`\ohorn ng is funded by Vietnam National Foundation for Science and Technology Development (NAFOSTED) under grant number 101.04-2015.26.}
\thanks{Sijong Kwak was supported by Basic Science Research Program through the National Research Foundation of Korea (NRF) funded by the Ministry of Science and ICT (2015R1A2A2A01004545).}


\begin{abstract}
In this paper, we prove the degree upper bound of projective subschemes in terms of the reduction number and show that the maximal cases are only arithmetically Cohen-Macaulay with linear resolutions. Furthermore, it can be shown that there are only two types of reduced, irreducible projective varieties with almost maximal degree. We also give the possible explicit Betti tables for almost maximal cases. In addition, interesting examples are provided to understand our main results.
\end{abstract}
\maketitle


\tableofcontents

\section{Introduction}

Let $X\subset \bP^{n+e}$ be a non-degenerate closed subscheme of dimension $n$ and codimension $e$ defined over an algebraically closed field $k$ of arbitrary characteristic with the ideal sheaf $\mathcal I_X$. Let $S_0=k[x_0, \ldots, x_{n+e}]$ and $R=S_0/I_X$ be the homogeneous coordinate ring  of $X$ where $I_X=\oplus_{m=0}^\infty H^0(\bP^{n+e}, \mathcal I_X(m))$ is the saturated homogeneous ideal. Among the numerical invariants of $X$ there are the degree $\deg(X)$, the Castelnuovo-Mumford regularity $\reg(R)$ and the reduction number $r(X)$ (which is defined as the reduction number of the homogeneous coordinate ring $R$). The complexity of $R$ is reflected in those invariants. The Castelnuovo-Mumford regularity is the height of the Betti table of $R$ and the reduction number $r(X)$ is the least number among the maximal degrees of minimal generators of $R$ as a module over its Noether normalizations.

There have been several results on the relations between these invariants (see \cite{NVT87}, \cite{Vas96}). For example, we always have
\begin{equation}\label{eq1}
1\leq r(X)\leq \reg(R).
\end{equation}
We can generalize the inequality (\ref{eq1}) as follows: If $R$ is $d$-regular until the $e$-th step in the minimal free resolution, then
\begin{equation}\label{eq2}
r(X)\leq d\leq \reg(R).
\end{equation}
It is interesting to mention that in many cases, the reduction number $r(X)$ is smaller than $d$ as shown in Examples \ref{34} (Ulrich's example), \ref{57} and \ref{58}.

On the other hand, the degree upper bound can be read off from the shape of Betti tables. As examples, projective varieties satisfying property $\textbf{N}_{2,p}$ for $p\ge 1$ have a degree bound (\cite{AS10}, \cite{HK12}). The multiplicity conjecture (\cite{HS98}), which was proved by Eisenbud and Schreyer, Boij and Söderberg (\cite{BS12}, \cite{ES09}, \cite{ES11}, \cite{Pe11}), gives us an upper bound with the maximal cases which are arithmetically Cohen-Macaulay with pure resolutions. More precisely, let $\beta_{ij}^{S_0}(R):=\dim_k\Tor_i^{S_0}(R, k)_{i+j}$ be the $(i,j)$-th graded Betti number and let $d_i:=\max\{i+j: \beta_{ij}^{S_0}(R)\not=0\}$ for $i=1, 2, \ldots, e$. We have a degree upper bound for projective subschemes, not necessarily arithmetically Cohen-Macaulay, namely,
$$\deg(X)\leq \frac{1}{e!}\prod_{i=1}^ed_i,$$
where equality holds if and only if $R$ is Cohen-Macaulay with a pure resolution.

Along this line, the reduction number also provides an upper bound for the degree of a projective subscheme $X\subset \bP^{n+e}$. Actually, in this paper we prove the following theorem.

\begin{theorem}\label{31}
Let $X\subset \bP^{n+e}$ be a non-degenerate closed subscheme of dimension $n$ and reduction number $r$. Then
\begin{equation}\label{eq3}
\deg(X)\leq \binom{e+r}{r}.
\end{equation}
Furthermore, $\deg(X)=\binom{e+r}{r}$ if and only if $X$ is arithmetically Cohen-Macaulay with an $(r+1)$-linear minimal free resolution.
\end{theorem}

The second part of Theorem \ref{31} has been proven by Ahn-Kwak \cite{AK15} under the assumption that $X$ is an algebraic set satisfying the $N_{3,e}$-property and $k$ is of characteristic zero. In the present paper, in place of $N_{d,e}$-property we use the reduction number to get a better degree upper bound and give a short proof for the characterization of closed subschemes of maximal degree.

If a projective closed subscheme satisfies the equivalent conditions in Theorem \ref{31}, i.e., $\deg(X)=\binom{e+r}{r}$, we say that it has a {\it maximal degree}. Combining Theorem \ref{31} and its proof with the description of the Betti table of Cohen-Macaulay graded $k$-algebra with linear resolution due to Eisenbud-Goto \cite[Proposition 1.7]{EG84}, we obtain the following consequence.

\begin{corollary}\label{32}
Let $X\subset \bP^{n+e}$ be a non-degenerate closed subscheme of dimension $n$, codimension $e$ and reduction number $r$. Let $I_X\subset S_0$ be the defining ideal of $X$ and $R=S_0/I_X$. The following statements are equivalent:
\begin{enumerate}[(a)]
\item $X$ has a maximal degree; 
\item Suppose the natural homomorphism $S=S_e=k[x_e, \ldots, x_{n+e}] \rightarrow R$ is a Noether normalization with the reduction number $r_S(R)=r$. Then the initial ideal $\ini(I_X)$ with respect to the degree reverse lexicographic order is generated by the set of all monomials in $x_0, x_1, \ldots, x_{e-1}$ of degree $r+1$;
\item The graded Betti numbers of $X$ are
$$\beta_{ij}(X):=\beta_{ij}(R)=\begin{cases}
1 &\mbox{ if } i=j=0,\\
\binom{e+r}{i+r}\binom{i-1+r}{r}&\mbox{ if } j=r, 1\le i\le e,\\
0 &\mbox{ otherwise.}
\end{cases}
$$
\end{enumerate}
\end{corollary}

We then investigate the `almost maximal' cases with $\deg(X)=\binom{e+r(X)}{r(X)}-1$. To begin with, we have the following refined inequality
\begin{equation}\label{eq4}
\deg(X)\leq \mu_S(R)\leq \binom{e+r(X)}{r(X)},
\end{equation}
where $S\hookrightarrow R$ is a Noether normalization which determines the reduction number $r(X)$, and $\mu_S(R)$ is the number of minimal generators of $R$ as an $S$-module. The almost maximal cases consist of arithmetically Cohen-Macaulay subschemes with $\mu_S(R)=\binom{e+r(X)}{r(X)}-1$ and non-arithmetically Cohen-Macaulay subschemes with $\mu_S(R)=\binom{e+r(X)}{r(X)}$. In this paper, we characterize both two cases by describing the structures of the initial ideals with respect to the degree reverse lexicographic order and the syzygies of the Noether normalization $S\hookrightarrow R$ (see Theorem \ref{41} and Corollary \ref{52}). 

As an application, we obtain explicit Betti tables of those subschemes in each case by analyzing the Betti tables of the initial ideal of $I_X$, then applying the graded mapping cone sequence (Theorem \ref{26}) and the Cancellation Principle due to M. Green \cite[Corollary 1.21]{Gre98} (see also \cite[Section 3.3]{HH11}). In the first case, we note that a projective subscheme $X$ is arithmetically Cohen-Macaulay if and only if $R$ is a free $S$-module, i.e., no syzygy among generators of $R$ over $S$ (see Corollary \ref{27}). The Betti table of an arithmetically Cohen-Macaulay subscheme of almost maximal degree is obtained in the following theorem.

\begin{theorem}\label{45}
Let $X\subset \bP^{n+e}$ be a non-degenerate closed subscheme of dimension $n$ and reduction number $r$. Assume that $X$ is arithmetically Cohen-Macaulay with $\deg(X)=\binom{e+r}{r}-1$. The Betti table of the homogeneous coordinate ring of $X$ is (in the following table, we write only rows with some possibly non-zero entries)

\newpage
\begin{figure}[!htb]
\begin{tabular}{>{\centering}m{2cm}|>{\centering}m{1cm} >{\centering}m{1cm} >{\centering}m{1cm} >{\centering}m{1cm} >{\centering}m{1cm} c}
			&	0&		1&	2&	$\cdots$&		$e-1$&	$e$\\
			\hline
		0	&	1&		--&	--&	$\cdots$&	--&			--\\
		r-1	&	--&	$1$&	--&$\cdots$&	-- &--\\
		r	&	--&	$\beta_{1r}$&	$\beta_{2r}$&$\cdots$&	$\beta_{e-1,r}$ &$\beta_{e,r}$\\
\end{tabular}
\end{figure}

\noindent where for $i=1, \ldots, e$, 
$$\beta_{i,r}=\binom{e+r}{i+r}\binom{r+i-1}{r}-\binom{e}{i}.$$
\end{theorem}

The case of non-arithmetically Cohen-Macaulay subschemes is more complicated as these subschemes might have big projective dimension. We restrict to the smaller category of projective varieties, i.e., reduced, irreducible projective subschemes and show that if $X$ is a projective variety of almost maximal degree then $\depth(R)\geq \dim(X)$ (Theorem \ref{51}). The assumption of being projective variety is actually crucial. Using this result on arithmetic depth, we are able to describe an initial ideal of $I_X$ and as in the previous cases, we obtain explicit Betti tables for varieties of almost maximal degree as in the following theorem.

\begin{theorem}\label{54}
Let $X\subset \bP^{n+e}$ be a non-degenerate projective variety of dimension $n$, codimension $e$ and reduction number $r$.  Let $R$ be the homogeneous coordinate ring of $X$. Suppose $\deg(X)= \binom{e+r}{r}-1$ and $X$ is not arithmetically Cohen-Macaulay. Then the Betti table (over $S_0$) of $R$ has one of the following shapes (in the following tables, we write only rows with some possibly non-zero entries):

\begin{enumerate}
\item[(a)] $\reg(R)=r$:

\begin{figure}[!htb]
\begin{tabular}{>{\centering}m{1.5cm}|>{\centering}m{1cm} >{\centering}m{1cm}>{\centering}m{1cm} >{\centering}m{1cm} >{\centering}m{1cm} >{\centering}m{1cm} c}
	&	$0$	&		$1$&	$\ldots$ &	$i$&	$\ldots$ &$e$&	$e+1$\\
\hline
$0$	&	$1$	&		--&		$\ldots$	&	--&		$\ldots$	&--&		--\\
$r$	&	--&	 	$\beta_{1,r}$		&	$\ldots$	&	$\beta_{i,r}$		&	$\ldots$	&$\beta_{e,r}$&		$1$\\
\end{tabular}
\end{figure}

\noindent where for $1\leq i\leq e+1$,
$$\beta_{i,r}=\binom{e+r}{i+r}\binom{r+i-1}{r}+\binom{e}{i-1}.$$

\item[(b)] $\reg(R)=r+1$:

\begin{figure}[!htb]
\begin{tabular}{>{\centering}m{1.5cm}|>{\centering}m{1cm} >{\centering}m{1cm}>{\centering}m{1cm} >{\centering}m{1cm} >{\centering}m{1cm} >{\centering}m{1cm} c}
	&	$0$	&		$1$&	$\ldots$ &	$i$&	$\ldots$ &$e$&	$e+1$\\
\hline
$0$	&	$1$	&		--&		$\ldots$	&	--&		$\ldots$	&--&		--\\
$r$	&	--&	 	$\beta_{1,r}$		&	$\ldots$	&	$\beta_{i,r}$		&	$\ldots$	&$\beta_{e,r}$&		--\\
$r+1$	&	--&	 	$\beta_{1,r+1}$		&	$\ldots$	&	$\beta_{i,r+1}$		&	$\ldots$	&$\beta_{e,r+1}$&		$1$\\
\end{tabular}
\end{figure}

\noindent where for $1\leq i\leq e+1$,
$$\beta_{i, r}-\beta_{i-1,r+1}=\binom{e+r}{i+r}\binom{r+i-1}{r}-\binom{e}{i-2}.$$

\item[(c)] $\reg(R)>r+1$:

\begin{figure}[!htb]
\begin{tabular}{>{\centering}m{1.5cm}|>{\centering}m{1cm} >{\centering}m{1cm}>{\centering}m{1cm} >{\centering}m{1cm} >{\centering}m{1cm} >{\centering}m{1cm} c}
	&	$0$	&		$1$&	$\ldots$ &	$i$&	$\ldots$ &$e$&	$e+1$\\
\hline
$0$	&	$1$	&		--&		$\ldots$	&--&		$\ldots$	&	--&		--\\
$r$	&	--&	 	$\beta_{1r}$		&	$\ldots$	&$\beta_{ir}$		&	$\ldots$	&	$\beta_{e,r}$&		--\\
$\reg(R)$	&	--&	 	$\binom{e}{0}$		&	$\ldots$	&$\binom{e}{i-1}$		&	$\ldots$	&	$\binom{e}{e-1}$&		$\binom{e}{e}$\\
\end{tabular}
\end{figure}

\noindent where for $1\leq i\leq e+1$,
$$\beta_{i,r}=\binom{e+r}{i+r}\binom{i+r-1}{r},$$
$$\beta_{i,\reg(R)}=\binom{e}{i-1}.$$
\end{enumerate}
\end{theorem}

We also provide interesting examples with all possible Betti tables to illustrate our main results. 

About the structure of this paper, we prove Theorem \ref{31} on degree upper bound with a characterization of the maximal cases and give further consequences of this theorem in Section 3. Projective subschemes of almost maximal degree are studied in the next two sections. Theorem \ref{45} on Betti tables of arithmetically Cohen-Macaulay subschemes of almost maximal degree is proved in Section 4. In Section 5 we focus on projective varieties of almost maximal degree which are not arithmetically Cohen-Macaulay. Theorem \ref{54} is proved in this section. 

In this paper, the Betti tables are computed by using Macaulay 2 \cite{GS}.

\medskip
\noindent{\bf Acknowledgments.} The authors thank the anonymous referee for careful reading and many useful comments which help to improve the presentation of the paper. The first author thanks
\fontencoding{T5}\selectfont Nguy\~\ecircumflex n \DJ\abreve ng H\d\ohorn p
for useful discussion on Macaulay 2 and graded Betti numbers of monomial ideals, in particular, for suggesting the short proof of Lemma \ref{44}. This work has been done during the visit of \fontencoding{T5}\selectfont \DJ o\`an Trung C\uhorn{}\`\ohorn ng to the Korea Advanced Institute of Science and Technology (KAIST). He thanks Professor Sijong Kwak and KAIST for support and hospitality during his visit. The second author would like to thank KAIST Grand Challenge 30 Project for financial support which has been the basis of his research activities.


\section{Preliminaries}
\subsection{Reduction number}
Throughout this paper, $k$ is an infinite field. Let $S_i=k[x_i, \ldots, x_{n+e}]$ be the polynomial ring over $k$, for $i=0, 1, \ldots, n+e$. Let $I\subset S_0$ be a homogeneous ideal of codimension $e$ and put $R=S_0/I$. The irrelevant homogeneous ideal of $R$ is denoted by $R_+$. Let $J$ be a minimal homogeneous reduction of $R_+$. The reduction number of $R_+$ with respect to $J$ is
$$r_J(R_+)=\min\{t>0: R_+^{t+1}=JR_+^t\}=\min\{t>0: R_{t+1}=J_{t+1}\}.$$
This number is a measure for the complexity of the algebra $R$. It has deep relations with other invariants of the same type such as the Castelnuovo-Mumford regularity and the a-invariant. The latter is defined by
$$a_{n+1}(R)=\max\{t: H^{n+1}_{R_+}(R)_t\not=0\}+1.$$
N.V. Trung gave the following comparison.

\begin{proposition}{\cite[Proposition 3.2]{NVT87}} \label{21} We have $a_{n+1}(R)+n\leq r_J(R_+)\leq \reg(R)$. In addition, if $R$ is Cohen-Macaulay, then $a_{n+1}(R)+n= r_J(R_+)=\reg(R)$.
\end{proposition}

On the other hand, the ideal $J$ is minimally generated by $n+1$ linear forms and after a change of variables, we may assume that $J=(x_e, \ldots, x_{n+e})$. Let $S=S_e=k[x_e, \ldots, x_{n+e}]$. The natural homomorphism $S\rightarrow R$ is in fact an inclusion and is a Noether normalization of $R$, i.e., $S$ is a polynomial $k$-algebra and $R$ is a finitely generated $S$-module. The reduction number can be computed by the maximal degree of minimal generators of $R$ over $S$. Namely,

\begin{lemma}{\cite[Proposition 5.1.3]{Vas94}}\label{22}
Let $b_1, \ldots, b_s$ be a minimal set of homogeneous generators of $R$ as an $S$-module, then
$$r_J(R_+)=\max\{\deg(b_1), \ldots, \deg(b_s)\}.$$
\end{lemma}

The reduction number of $R$ is the least number among $r_J(R_+)$ for all minimal reduction $J$ of $R_+$ and will be denoted by $r(R)$.

Sometimes we also call $r_J(R_+)$ the reduction number of the Noether normalization $S\hookrightarrow R$ and denote it by $r_S(R)$. The finite $S$-algebra structure on $R$ is particularly interesting and should provide an effective way to understand the structure of the algebra. We have some very first properties of minimal sets of generators of $R$ over $S$.

\begin{proposition}\label{23}
The minimal number of generators of $R$ over $S$ is bounded by
$$\mu_S(R)\leq \binom{e+r_S(R)}{r_S(R)}.$$
\end{proposition}
\begin{proof}
As an $S$-module, $R$ is generated by all monomials in $x_0, \ldots, x_{e-1}$ of degree at most $r_S(R)$. There are $\binom{e+r_S(R)}{r_S(R)}$ such monomials and the inequality follows.
\end{proof}

We can describe precisely a minimal set of monomials generating $R$ considered as an $S$-module. To do this, we  fix the degree reverse lexicographic order on the monomial of $S_0$. Then $R$ is minimally generated over $S$ by monomials which formulate a basis of the $k$-vector space $S_0/I+(x_e, \ldots, x_{n+e})$. These are exactly the standard monomials with respect to the ideal $I_X+(x_e, \ldots, x_{n+e})$. Recall that a monomial in $S_0$ is a standard monomial with respect to a homogeneous ideal $J$ if it is not contained in the initial ideal of $J$. We sum up in the following lemma.

\begin{lemma}\label{24}
Let $S=S_e=k[x_e, \ldots, x_{n+e}]$ and suppose that the natural map $S\rightarrow R$ is a Noether normalization of $R$. Then $R$ is minimally generated as a module over $S$ by the set
$$\{\overline{m}\in R: m\in S_0 \text{ is a monomial in $x_0, \ldots, x_{e-1}$ which is standard with respect to } I_X+(x_e, \ldots, x_{n+e})\}.$$
\end{lemma}

Related to this lemma, the following criterion for Cohen-Macaulayness is very useful in the next sections.

\begin{proposition}\label{25}
Let $I\subset S_0$ be a homogeneous ideal and $R=S_0/I$. Suppose $S=S_e=k[x_e, \ldots, x_{n+e}]\rightarrow R$ is a Noether normalization. We fix the degree reverse lexicographic order on the monomial of $S_0$. Then 
$$\deg(R)\leq \mu_S(R).$$
Furthermore, the following statements are equivalent:
\begin{enumerate}[(a)]
\item $R$ is Cohen-Macaulay;
\item $\deg(R)=\mu_S(R)$;
\item The initial ideal $\ini(I)$ is minimally generated by a set of monomials in $x_0, \ldots, x_{e-1}$.
\end{enumerate}
\end{proposition}

It should be remarked that the equivalence of $(a)$ and $(c)$ is due to  Bermejo-Gimenez \cite[Proposition 2.1]{BG01}. The inequality in Propoposition \ref{25} can be seen as a graded  Lech's inequality between length and degree. The equivalence between $(a)$ and $(b)$ should be well-known to experts. It could be proved by localizing things with respect to the maximal homogeneous ideal, then use \cite[Theorem 17.11]{Mat86}. Belows we present another proof using initial ideal and Bermejo-Gimenez's result.

\begin{proof}
We denote the minimal monomial generators of $\ini(I)$ by 
$$w_1, \ldots, w_t, u_1v_1, \ldots, u_sv_s,$$
where $w_1, \ldots, w_t, u_1, \ldots, u_s$ are monomials in $x_0, \ldots, x_{e-1}$ of positive degrees and $v_1, \ldots, v_s$ are monomials in $x_e, \ldots, x_{n+e}$, also of positive degrees. Following Lemma \ref{24}, $\mu_S(R)$ is the number of monomials in $x_0, \ldots, x_{e-1}$ not lying in the ideal $(w_1, \ldots, w_t)$. On the other hand, the only minimal prime ideal of $\ini(I)$ is $(x_0, \ldots, x_{e-1})$. Hence
$$\deg(R)=\deg(S_0/\ini(I))=\deg(S_0/(w_1, \ldots, w_t, u_1, \ldots, u_s)).$$
The latter number, as $S_0/(w_1, \ldots, w_t, u_1, \ldots, u_s)$ is Cohen-Macaulay, is the number of monomials in $x_0, \ldots, x_{e-1}$ not lying in $(w_1, \ldots, w_t, u_1, \ldots, u_s)$. Hence
$$\deg(R)\leq \mu_S(R),$$
and equality occurs if and only if $s=0$, or in other words, $\ini(I)$ is minimally generated by some monomials in $x_0, \ldots, x_{e-1}$. The latter holds true if and only if $R$ is Cohen-Macaulay due to $(a)\Leftrightarrow (c)$.
\end{proof}


\subsection{Syzygies and Betti table}

Let $M$ be a finitely generated graded $S_0$-module. The minimal free resolution of $M$ is
$$\cdots \rightarrow F_i=\bigoplus_j S_0(-i-j)^{\beta^{S_0}_{ij}}\rightarrow \cdots  \rightarrow F_0=\bigoplus_j {S_0}(-j)^{\beta^{S_0}_{0,j}}\rightarrow 0,$$
where $\beta^{S_0}_{ij}=\dim_k\Tor_i^{S_0}(M, k)_{i+j}$ is the $(i, j)$-th graded Betti number. The number $\beta^{S_0}_i(M)=\mathrm{rank}_{S_0}(F_i)=\sum_{j\in \bZ}\beta_{ij}^{S_0}(M)$ is called the $i$-th Betti number of $M$. The module $M$ is $d$-regular if $\beta_{ij}^{S_0}(M)=0$ for all $i\geq 0$ and all $j> d$. The Castelnuovo-Mumford regularity of $M$ is
$$\reg(M)=\min\{d: M \text{ is } d-\text{regular}\}.$$
The Castelnuovo-Mumford regularity together with the projective dimension are the height and the width of the Betti table.

We say that $R$ has a $d$-linear resolution if $\beta_{ij}^{S_0}(R)=0$ unless $i=j=0$ or $j=d-1$ and $i\geq 1$. Obviously, if $R$ has a $d$-linear resolution then $I$ is generated by a set of forms of degree $d$. 

Closely related to the regularity is the $N_{d,p}$-property defined by Eisenbud-Goto \cite{EG84} and Eisenbud-Green-Hulek-Popescu \cite{EGHP05}. We say that $M$ satisfies the $N_{d,p}$-property ($d\geq 2$) if $\beta_{i,j}^{S_0}(M)=0$ for all $i\leq p$ and $j\geq d$. In other words, $M$ satisfies the $N_{d,p}$-property if $M$ is $(d-1)$-regular up to degree $p$. Clearly $M$ satisfies the $N_{d,p}$-property for all $d>\reg(M)$ and all $p\geq 0$.

In the study of the structure of modules with $N_{d,p}$-property, a mapping cone sequence of homology groups has been used effectively. Recall the notation $S_i=k[x_i, \ldots, x_{n+e}]$, $i=1, 2, \ldots, n+e$. Through the inclusion $S_1\subset S_0$, any graded $S_0$-module $M$ is also a graded $S_1$-module. The Koszul complex $K(x_1, \ldots, x_{n+e}; M)$ fits in a short exact sequence of complexes
$$0\rightarrow K(x_1, \ldots, x_{n+e}; M) \rightarrow K(x_1, \ldots, x_{n+e}; K(x_0; M))\rightarrow K(x_1, \ldots, x_{n+e}; M[-1])[-1]\rightarrow 0,$$
where $K(x_0, M)$ is the Koszul complex of $M$ with respect to the single element $x_0$. Then there is a long exact sequence of Koszul homology
\begin{multline*}
\cdots \rightarrow H_i(x_1, \ldots, x_{n+e}; M)_{i+j}\rightarrow H_i(x_0, \ldots, x_{n+e}; M)_{i+j}\rightarrow H_{i-1}(x_1, \ldots, x_{n+e}; M)_{i+j-1}\\
\stackrel{*x_0}\longrightarrow H_{i-1}(x_1, \ldots, x_{n+e}; M)_{i+j}\rightarrow H_{i-1}(x_0, \ldots, x_{n+e}; M)_{i+j}\rightarrow \cdots
\end{multline*}

We have $H_i(x_1, \ldots, x_{n+e}; M)_{i+j}\simeq \Tor^{S_1}_i(M,k)_{i+j}$ and $H_i(x_0, \ldots, x_{n+e}; M)_{i+j}\simeq \Tor^{S_0}_i(M,k)_{i+j}$. This leads to a so-called graded mapping cone sequence obtained by Ahn-Kwak.

\begin{theorem}{\cite[Theorem 3.2]{AK11}}\label{26}
Let $M$ be a graded $S_0$-module. There is a long exact sequence
\begin{multline}\label{eq5}
\cdots \rightarrow \Tor^{S_1}_i(M,k)_{i+j}\rightarrow \Tor^{S_0}_i(M,k)_{i+j}\rightarrow \Tor^{S_1}_{i-1}(M,k)_{i+j-1}\\
\stackrel{\delta}\longrightarrow \Tor^{S_1}_{i-1}(M,k)_{i+j}\rightarrow \Tor^{S_0}_{i-1}(M,k)_{i+j}\rightarrow \cdots
\end{multline}
whose connecting homomorphisms $\delta$'s are induced from the multiplication by $x_0$.
\end{theorem}

The following consequence follows immediately from the exact sequence in the theorem.

\begin{corollary}\label{27}
\begin{enumerate}
\item[(a)] Keep the notations as in Theorem \ref{26}. Denote
$$\chi_m^{S_0}(M)=\sum_{j=0}^m(-1)^j\beta_{m-j, j}^{S_0}(M),$$
for $m\in \bZ$. Then we have an additive formula
$$\chi_m^{S_0}(M)=\chi_m^{S_1}(M)+\chi_{m-1}^{S_1}(M).$$
Furthermore,
$$\chi_m^{S_0}(M)=\sum_{j=0}^e\binom{e}{j}\chi_{m-j}^{S_e}(M).$$
\item[(b)] Fix two indexes $p>0$ and $d\geq 0$. If $\beta_{ij}^{S_1}(M)=0$ for all $i\geq p$ and $j=0,1, \ldots, d$, then $\beta_{i+1, j}^{S_0}(M)=0$ for all $i\geq p$ and $j=0,1, \ldots, d$;
\item[(c)] \cite[Corollary 2.2]{AK15} $\mathrm{proj.dim}_{S_0}(M)=\mathrm{proj.dim}_{S_t}(M)+t$;
\item[(d)] \cite[Proposition 2.3]{AK15} Fix an index $i$. If $\beta_{ij}^{S_0}(M)=0$ for all $j\geq d$, then $\beta_{i-1, j}^{S_1}(M)=0$ for all $j\geq d$. Consequently, if $M$ satisfies the $N_{d, p}$-property over $S_0$ then $M$ satisfies the $N_{d, p-t}$-property over $S_t$ for $t>0$. Furthermore, we have $\reg_{S_0}(M)=\reg_{S_t}(M)$.
\end{enumerate}
\end{corollary}
\begin{proof}
(a) and (b) are induced directly from Theorem \ref{26} and (c) is a consequence of $(b)$.

In order to prove (d), we note that if $\beta_{ij}^{S_0}(M)=0$ then from the exact sequence (\ref{eq5}) in Theorem \ref{26}, $\beta_{i-1, j}^{S_1}(M)\leq \beta_{i-1, j+1}^{S_1}(M)$. If $\beta_{ij}^{S_0}(M)=0$ for all $j\geq d$ then
$$\beta_{i-1, d}^{S_1}(M)\leq \beta_{i-1, d+1}^{S_1}(M)\leq \ldots \leq \beta_{i-1, j}^{S_1}(M)\leq \ldots $$
But $\beta_{i-1, j}^{S_1}(M)=0$ for large enough $j$, thus
$$\beta_{i-1, d}^{S_1}(M)=\beta_{i-1, d+1}^{S_1}(M)=\ldots=0.$$
\end{proof}

As an application of Corollary \ref{27}, we obtain an upper bound for the reduction number which is somehow stronger than Trung's upper bound in Proposition \ref{21}.

\begin{corollary}\label{28}
Let $I\subset S_0$ be a homogeneous ideal of codimension $e$. Suppose $R=S_0/I$ satisfies the $N_{d,e}$-property. Then
$$ r(R)<d $$
where $r(R)$ is the reduction number.
\end{corollary}
\begin{proof}
It should be mentioned that $R$ always satisfies the $N_{d,e}$-property for some $d\leq \reg(R)$. So this corollary gives a stronger upper bound than Trung's.

We may assume that the natural homomorphism $S_e=k[x_e, \ldots, x_{n+e}]\rightarrow R$ is a Noether normalization of $R$. By Corollary \ref{27}(d), $R$ satisfies the $N_{d,0}$-property as an $S_e$-module. In other words, $\beta_{0,j}^{S_e}(R)=0$ for all $j\geq d$ and hence $r(R)<d$ (see Lemma \ref{22}).
\end{proof}


$\blacksquare$ Notations and Conventions \\
\medskip
Let $X\subset \bP^{n+e}$ be a non-degenerate closed subscheme of dimension $n$ with the ideal sheaf $\mathcal I_X$.
Let $I_X=\oplus_{m=0}^\infty H^0(\bP^{n+e}, \mathcal I_X(m))$ be the saturated homogeneous ideal and $R=k[x_0, \ldots, x_{n+e}]/I_X$ be the homogeneous coordinate ring.
\begin{itemize}
\item The reduction number of $X$ is the same as the reduction number of $R$, i.e., $r(X)=r(R)$.

\item The $(i,j)$-th graded Betti number of $X$ is
$$\beta_{i,j}(X):=\beta_{i,j}^{S_0}(R)=\beta^{S_0}_{i-1,j+1}(I_X)$$

\item The Castelnuovo-Mumford regularity of $X$ is $\reg(X)=\reg(I_X)=\reg(R)+1$.

\item We say that $X$ satisfies the $N_{d,p}$-property if so does $R$. In other words, $X$ satisfies the $N_{d,p}$-property if $R$ is $(d-1)$-regular up to the degree $p$.

\item A projective variety is always assumed to be an irreducible and reduced projective subscheme.
\end{itemize}

Let $X\subset\bP^{n+e}$ be a projective variety of codimension $e$. We say that $X$ has the minimal degree if $\deg(X)=e+1$. The variety $X$ is a del Pezzo variety if it is arithmetically Cohen-Macaulay with $\deg(X)=e+2$ (almost minimal degree).

In the whole paper, we only consider the degree reverse lexicographic order on the monomials. The initial ideal of a homogeneous ideal $I$ with respect to this order is denoted by $\ini(I)$.


\section{Degree upper bound in terms of reduction number and the maximal cases}

The nature of the reduction number is a bound for the complexity of an algebra or the associated scheme as we can see in Proposition \ref{21} and especially Lemma \ref{22}. The  upper bound for the degree in terms of the reduction number in Theorem \ref{31} provides more evidence for this observation. In this section we first prove this theorem and then give several consequences and applications.

\begin{proof}[\bf Proof of Theorem \ref{31}]
Let $I_X\subset S_0$ be the saturated homogeneous defining ideal of $X$ and $R=S_0/I_X$. Changing the variables if necessary, we may assume that $Q=(x_e, \ldots, x_{n+e})$ is a minimal reduction of the irrelevant ideal $R_+$ with the reduction number $r_Q(R)=r$. In particular, $S=S_e=k[x_e, \ldots, x_{n+e}]\hookrightarrow R$ is a Noether normalization of $R$. Let $\mu_S(R)=\dim_k(R/S_+R)$ be the minimal number of generators of $R$ over $S$. Then we have
\begin{equation}\label{eq6}
\deg(X)\leq \mu_S(R)\leq \binom{e+r}{r},
\end{equation}
due to Propositions \ref{23} and \ref{25}. This proves the inequality (\ref{eq3}) in Theorem \ref{31}.

Now, if $X$ is arithmetically Cohen-Macaulay with an $(r+1)$-linear resolution then Eisenbud-Goto \cite[Proposition 1.7]{EG84} showed that
$$\deg(X)= \binom{e+r}{r}.$$

Conversely, assume that $\deg(X)=\binom{e+r}{r}$. Then all inequalities in (\ref{eq6}) are actually equalities, i.e.,
$$\mu_S(R)=\deg(X)=\binom{e+r}{r}.$$
Then $R$ is Cohen-Macaulay by using Proposition \ref{25} again.

It remains to show that $X$ has an $(r+1)$-linear resolution. Since $R$ is Cohen-Macaulay, we have by Proposition \ref{21} that
$$\reg(R)=a_{n+1}(R)+n=r_S(R)=r.$$
Hence $\reg(I_X)=\reg(R)+1=r+1$. On the other hand, combining the fact $\mu_S(R)=\binom{e+r}{r}$ with Lemma \ref{24}, we can see that $R$ is minimally generated over $S$ by all the monomials $x_0^{\alpha_0}\ldots x_{e-1}^{\alpha_{e-1}}$ with $\alpha_0+\ldots+\alpha_{e-1}\leq r$. They are all standard monomials with respect to $I_X+(x_e, \ldots, x_{n+e})$ and the degree reverse lexicographic order on the monomials. Note also that, due to Bermejo-Gimenez (see Proposition \ref{25}), $\ini(I_X)$ has a set of generators consisting of monomials in $x_0, \ldots, x_{e-1}$. Therefore
$$\ini(I_X)=(x_0^{\alpha_0}\ldots x_{e-1}^{\alpha_{e-1}}: \alpha_0+\ldots+\alpha_{e-1}= r+1).$$
In particular, the minimal degree of generators of $I_X$ is $r+1$. Combining this with the fact that $\reg(I_X)\leq r+1$, we get that $I_X$ is generated by a set of forms of degree $r+1$, hence $R$ has an $(r+1)$-linear minimal free resolution.
\end{proof}

\begin{example}\label{32-1}({\bf{Zero dimensional schemes with maximal degree}})
Let $\Gamma \subset \bP^e$ be a non-degenerate finite scheme with the reduction number $r$. Since $\Gamma$ is arithmetically Cohen-Maculay, the reduction number $r$ is equal to $\reg(R_{\Gamma})$. Therefore, it is $r$-normal and $\deg(\Gamma)\le \binom{e+r}{r}$ (or using Theorem \ref{31}). In particular, $\deg(\Gamma)=\binom{e+r}{r}$ if and only if there is no hypersurface of degree $r$ containing $\Gamma $ if and only if $R_{\Gamma}$ has an $(r+1)$-linear resolution.

For example, consider two distinct sets of $6$ points in $\bP^2$. One is $\Gamma_1$ not being contained in a conic curve. In this case, $r(\Gamma_1)=2$ and for a Noether normalization $S\rightarrow R_{\Gamma_1}$ which defines the reduction number $r(\Gamma_1)$, we have
$$S(-2)^{\oplus 3}\oplus S(-1)^{\oplus 2} \oplus S\simeq R_{\Gamma_1}.$$

The other is $\Gamma_2$ being contained in a unique conic curve. In this case, $r(\Gamma_2)=3$ and for a Noether normalization $S\rightarrow R_{\Gamma_2}$ which defines the reduction number $r(\Gamma_2)$, we also have
$$S(-3)\oplus S(-2)^{\oplus 2}\oplus S(-1)^{\oplus 2} \oplus S\simeq R_{\Gamma_2}.$$
\end{example}

\begin{example}\label{35}
There are many examples of projective subschemes with maximal degree (see Corollary \ref{32}). For examples, the algebraic set defined by the ideal of maximal minors of a $1$-generic $d \times(e +d -1)$-matrix of linear forms is such a subscheme which is arithmetically Cohen-Macaulay with $d$-linear resolution.
\end{example}

On the other hand, the proof of Theorem \ref{31} and Corollary \ref{32} induce the following combinatorial identity which is presented here for later usage.

\begin{lemma}\label{33}
Let $e, m, r$ be non-negative integers such $e>0$ and $m\leq r+e$. We have
$$\sum_{j=m-r}^e(-1)^{j}\binom{e}{j}\binom{e+m-j-1}{e-1}=
\begin{cases}
(-1)^{m+r}\binom{e+r}{m}\binom{m-1}{r} &\mbox{ if } m>r,\\
0 &\mbox{ if } 0<m\leq r,
\end{cases}
$$
where, by convention, $\binom{a}{b}=0$ if $a<b$.
\end{lemma}
\begin{proof}
Let $I\subset S_0=k[x_0, \ldots, x_{n+e}]$ be a monomial ideal generated by all monomials  in $x_0, \ldots, x_{e-1}$ of degree $r+1$. Then $R=S_0/I$ is a Cohen-Macaulay $k$-algebra with an $(r+1)$-linear minimal free resolution and $\deg(R)=\binom{e+r}r$ (see Corollary \ref{32}). Moreover, $(x_e, \ldots, x_{n+e})$ is a minimal reduction of $R_+$ with the reduction number $r$. In particular, $S=S_e=k[x_e, \ldots, x_{n+e}]\subset R$ is a Noether normalization of $R$.

As an $S$-module, the Betti numbers of $R$ are
$$\beta_{ij}^S(R)=\begin{cases}
\binom{e+j-1}{j} &\mbox{ if } i=0, 0\leq j\leq r,\\
0 &\mbox{ otherwise}.
\end{cases}$$

On the other hand, the Betti numbers of the $S_0$-module $R$ are
$$\beta_{ij}^{S_0}(R)=\begin{cases}
1&\mbox{ if } i=j=0,\\
\binom{e+r}{i+r}\binom{i+r-1}{r} &\mbox{ if } 0\leq i\leq e, j=r,\\
0 &\mbox{ otherwise}.
\end{cases}$$
due to \cite[Proposition 1.7]{EG84}.

Following Corollary \ref{27}, the Betti numbers of $R$ over $S$ and over $S_0$ are related by the equality
$$\chi_m^{S_0}(R)=\sum_{j=0}^e\binom{e}{j}\chi_{m-j}^{S}(R),$$
where
$$\chi_m^{S_i}(R)=\sum_{j=0}^m(-1)^j\beta_{m-j, j}^{S_i}(R),$$
for $m\in \bZ$. From the Betti table of $R$ as an $S$-module above, we have
$$\chi_{m-j}^{S}(R)=(-1)^{m-j}\beta_{0,m-j}^S(R)=
\begin{cases}
(-1)^{m-j}\binom{e+m-j-1}{m-j}&\mbox{ if } j\geq m-r,\\
0&\mbox{ if } j<m-r.
\end{cases}
$$
Therefore, if $0<m\leq r$ then
$$\sum_{j=m-r}^e(-1)^j\binom{e}{j}\binom{e+m-j-1}{e-1}=0,$$
and if $r< m\leq r+e$,
$$\sum_{j=m-r}^e(-1)^j\binom{e}{j}\binom{e+m-j-1}{e-1}=(-1)^{m+r}\binom{e+r}{m}\binom{m-1}{r}.$$
\end{proof}

In general, the reduction number is smaller than the regularity index $d$ in the $N_{d,e}$-property. The following example shows that the reduction number is $1$ and the regularity index can be arbitrarily large.

\begin{example}(due to B. Ulrich)\label{34}
Let $S_0=k[x_0, x_1, x_2, x_3]$. Consider the projective non-reduced line defined by the ideal $I=((x_0, x_1)^2, x_0x_2^t+x_1x_3^t)$ for some $t\geq 2$. Let $R=S_0/I$. Then $(x_2, x_3)$ is a minimal reduction of $R_+$ and $S=k[x_2, x_3]\hookrightarrow R$ is a Noether normalization. Since $R=S+Sx_0+Sx_1$, the reduction number $r(R)=1$.

On the other hand, the Betti table of $R$ as an $S_0$-module is given by

\newpage
\begin{figure}[!htb]
\begin{tabular}{>{\centering}m{1.5cm}|>{\centering}m{1cm} >{\centering}m{1cm}  >{\centering}m{1cm} c}
	&	0&		1&	2&	$3$\\
\hline
0	&	1&		--&		--&		--\\
1	&	--&	 	$3$&		$2$&		--\\
2	&	--&	 	--&		--&		--\\
$\vdots$	&	$\vdots$&	 $\vdots$&	$\vdots$&		$\vdots$\\
$t-1$	&	--&	 	--&		--&		--\\
$t$	&	--&	 	$1$&	$2$&	$1$\\
\end{tabular}
\newline

where the symbol $-$ means that the corresponding Betti number is zero.
\end{figure}

\noindent So $\reg(R)=t$, $\mathrm{proj.dim}_{S_0}(R)=3$, $\depth(R)=1$. The Hilbert polynomial of $R$ is $P_R(n)=2n+t-1$ and thus $\deg(R)=2$.

\end{example}

In the second part of this section, we will apply Theorem \ref{31} and its consequences to study projective varieties with small reduction number.
Recall that a projective variety $X\subset\bP^{n+e}$ of codimension $e$ has a minimal degree if $\deg(X)=e+1$.

\begin{corollary}\label{36}
Let $X\subset \bP^{n+e}$ be a non-degenerate projective variety of dimension $n$ and codimension $e$. The following statements are equivalent:
\begin{enumerate}
\item[(a)] $X$ has a minimal degree;
\item[(b)] The reduction number $r(X)=1$;
\item[(c)] The Castelnuovo-Mumford regularity $\reg(X)=2$.
\end{enumerate}
\end{corollary}
\begin{proof}
Let $X$ be a variety with reduction number $r(X)=1$. The degree of $X$ is bounded above by $e+1$ due to Proposition \ref{23}. Hence $\deg(X)=e+1$ and $X$ has a minimal degree. Furthermore, $\reg(X)=2$ by \cite{EG84}. This shows $(b)\Rightarrow (a)$ and $(a)\Rightarrow (c)$.

On the other hand, if $\reg(X)=2$ then $\reg(R)=1$, where $R$ is the homogeneous coordinate ring of $X$. So the implication $(c)\Rightarrow (b)$ follows immediately from Proposition \ref{21}.
\end{proof}

The case of projective closed subschemes of reduction number $2$ is much more complicated. The del Pezzo varieties are interesting examples of such subschemes. Recall that a projective variety $X\subset\bP^{n+e}$ is a del Pezzo variety if it is arithmetically Cohen-Macaulay with $\deg(X)=\mathrm{codim}(X)+2$ (i.e. almost minimal degree).

\begin{example}\label{37}
Let $X\subset \bP^{n+e}$ be a del Pezzo variety of dimension $n$. The Betti table of the homogeneous coordinate ring of $X$ is

\begin{figure}[!htb]
\begin{tabular}{>{\centering}m{2cm}|>{\centering}m{1cm} >{\centering}m{1cm} >{\centering}m{1cm} >{\centering}m{1cm} >{\centering}m{1cm} c}
			&	0&		1&	2&	$\cdots$&		$e-1$&	$e$\\
			\hline
		0	&	1&		--&	--&	$\cdots$&	--&			--\\
		1	&	--&	$\beta_{11}$&	$\beta_{21}$&$\cdots$&	$\beta_{e-1,1}$ &--\\
		2	&	--&	--&	--&$\cdots$&	-- &$\beta_{e,2}=1$\\
\end{tabular}
\end{figure}

\noindent where $\beta_{i1}=i\binom{e+1}{i+1}-\binom{e}{i-1}$ (see \cite{Hoa93} and also the next proposition).

So $X$ satisfies the $N_{3,e}$-property but not the $N_{2, e}$-property. The reduction number of $X$ is at most $2$ by Corollary \ref{28}. On the other hand, as a del Pezzo variety does not have a minimal degree, its reduction number is at least $2$ by Corollary \ref{36}. This concludes that $r(X)=2$.
\end{example}

The next proposition provides information on Betti table of any arithmetically Cohen-Macaulay subscheme of reduction number $2$.

\begin{proposition}\label{38}
Let $X\subset \bP^{n+e}$ be a non-degenerate closed subscheme of dimension $n$. Assume that $X$ is arithmetically Cohen-Macaulay with reduction number $r(X)=2$. Then the Betti table of the homogeneous coordinate ring of $X$ (over $S_0$) is

\begin{figure}[!htb]
\begin{tabular}{>{\centering}m{2cm}|>{\centering}m{1cm} >{\centering}m{1cm} >{\centering}m{1cm} >{\centering}m{1cm} >{\centering}m{1cm} c}
			&	0&		1&	2&	$\cdots$&		$e-1$&	$e$\\
			\hline
		0	&	1&		--&	--&	$\cdots$&	--&			--\\
		1	&	--&	$\beta_{11}$&	$\beta_{21}$&$\cdots$&	$\beta_{e-1,1}$ &$\beta_{e,1}$\\
		2	&	--&	$\beta_{12}$&	$\beta_{22}$&$\cdots$&	$\beta_{e-1,2}$ &$\beta_{e,2}$\\
\end{tabular}
\end{figure}

\noindent where $$\beta_{i,2}=\beta_{i+1,1}+\binom{e}{i}\deg(X)-(i+1)\binom{e+2}{i+2}.$$
In particular,
$$\deg(X)=e+1+\beta_{e,2}.$$
\end{proposition}
\begin{proof}
Let $R$ be the homogeneous coordinate ring of $X$. By changing the variables, we might assume that $S=k[x_e, \ldots, x_{n+e}]$ is a Noether normalization of $R$ and reduction number $r_S(R)=r(R)=2$.

By the assumption, the Castelnuovo-Mumford regularity of $R$ is $2$ and almost all Betti numbers of $R$ over $S_0$ vanish except $\beta_{00}^{S_0}(R)=1$ and $\beta_{ij}^{S_0}(R)$ for $i=1, 2, \ldots, e$ and $j=1, 2$. Corollary \ref{27}(d) then implies that $R$ as an $S$-module has the Betti numbers $\beta_{ij}^{S}(R)=0$ except $\beta_{00}^S(R)=1$, $\beta_{01}^S(R)=e$ and
$$\beta_{02}^S(R)=\deg(X)-\beta_{00}^S(R)-\beta_{01}^S(R)=\deg(X)-e-1.$$

Now using the notations and results in Corollary \ref{27}(a), we have
$$\chi_m^{S_0}(R)=\sum_{j=0}^m(-1)^j\beta_{m-j, j}^{S_0}(R)=\beta_{m-2, 2}^{S_0}(R)-\beta_{m-1,1}^{S_0}(R),$$
$$\chi_{m}^S(R)=\sum_{j=0}^{m}(-1)^j\beta_{m-j, j}^S(R)=(-1)^m\beta_{0,m}(R),$$
and
$$\beta_{i, 2}^{S_0}(R)-\beta_{i+1,1}^{S_0}(R)=\binom{e}{i+2}-\binom{e}{i+1}\beta_{01}^S(R)+\binom{e}{i}\beta_{02}^{S}(R).$$
Let $\beta_{ij}:=\beta_{ij}^{S_0}(R)$, then we obtain
\[\begin{aligned}
\beta_{i2}&=\beta_{i+1,1}+\binom{e}{i+2}-\binom{e}{i+1}e+\binom{e}{i}(\deg(X)-e-1)\\
&=\beta_{i+1,1}+\binom{e}{i}\deg(X)-(i+1)\binom{e+2}{i+2}.
\end{aligned}\]
\end{proof}


\section{The arithmetically Cohen-Macaulay subschemes with almost maximal degree}

Let $X\subset \bP^{n+e}$ be a non-degenerate closed subscheme of dimension $n$ and the reduction number $r$. The degree of $X$ is always bounded above by $\binom{e+r}{r}$ and we have seen in Theorem \ref{31} a characterization for those subschemes whose degree attains the maximal value. In the next two sections we will investigate the almost maximal cases, namely, when
$$\deg(X)= \binom{e+r}{r}-1.$$

We keep the notations $S_i, I_X, R=S_0/I_X$ as in the previous section. Changing the variables if necessary, we assume that $J=(x_e, \ldots, x_{n+e})$ is a minimal reduction of $(R)_+$ with reduction number $r_J((R)_+)=r(R)=r$. Then $S=S_e=k[x_e, \ldots, x_{n+e}]$ is a Noether normalization of $R$. By Propositions \ref{23} and \ref{25}, there are bounds for the minimal number of generators of $R$ as an $S$-module
\begin{equation}\label{eq41}
\deg(X)\leq \mu_S(R)\leq \binom{e+r}{r}.
\end{equation}
If $\deg(X)=\binom{e+r}{r}-1$ then there are only two possibilities for $\mu_S(R)$, namely,
\begin{enumerate}
\item[1)] $\mu_S(R)= \binom{e+r}{r}-1$: the arithmetically Cohen-Macaulay case (see Proposition \ref{25});
\item[2)] $\mu_S(R)= \binom{e+r}{r}$: the non-arithmetically Cohen-Macaulay case.
\end{enumerate}

This section is devoted to study the arithmetically Cohen-Macaulay case with a proof of Theorem \ref{45}, while the latter will be considered in the next section.

Arithmetically Cohen-Macaulay projective subschemes of almost maximal degree are characterized in the following theorem.

\begin{theorem}\label{41}
Let $X\subset \bP^{n+e}$ be a non-degenerate closed subscheme of dimension $n$ and reduction number $r$. Let $I_X\subset S_0$ be the defining ideal of $X$ and $R=S_0/I_X$. Assume that the natural homomorphism $S=S_e=k[x_0, \ldots, x_{e-1}]\rightarrow R$ is a Noether normalization with the reduction number $r=r_S(R)$. The following statements are equivalent:
\begin{enumerate}
\item[(a)] $\deg(X)=\mu_S(R)= \binom{e+r}{r}-1$;
\item[(b)] $X$ is arithmetically Cohen-Macaulay, $\dim_k(I_X)_r=1$ and the truncated ideal $(I_X)_{\ge r+1}$ has a linear minimal resolution;
\item[(c)] The initial ideal  $\ini(I_X)$ with respect to the degree reverse lexicographic order is generated by a set of monomials in $x_0, x_1, \ldots, x_{e-1}$ consisting of all monomials of degree $r+1$ and a monomial of degree $r$.
\end{enumerate}
\end{theorem}
\begin{proof}
\noindent $(a) \Rightarrow (b)$: We have $\deg(X)=\mu_S(R)$ and hence $R$ is Cohen-Macaulay by Proposition \ref{25}. Then the Castelnuovo-Mumford regularity of $R$ is the same as the reduction number $r$ (see Proposition \ref{21}). This shows that $(I_X)_{>r}$ has a linear resolution.

Again, we fix the degree reverse lexicographic order on the monomials of $S_0$. There are totally $\mu_S(R)=\binom{e+r}{r}-1$ standard monomials with respect to $\ini(I_X)+(x_e, \ldots, x_{n+e})$ and they all have degrees from $0$ to $r$. This shows that there is a monomial $m$ in $x_0, \ldots, x_{e-1}$ with $\deg(m)\leq r$ which are not contained in the ideal $\ini(I_X)$. Then clearly $\deg(m)=r$ and $(I_X)_r=kg$ for some polynomial $g\in I(X)$ with $\ini(g)=m$.
\medskip

\noindent $(b)\Rightarrow (a)$: Since $X$ is arithmetically Cohen-Macaulay, we have by Propositions \ref{23} and \ref{25},
$$\deg(X)=\mu_S(R)\leq \binom{e+r}{r}.$$
The fact $\dim_k(I_X)_r=1$ implies that the $k$-vector space of degree $r$ component of the initial ideal of $I_X$ has dimension one too. Hence $\mu_S(R)\geq \binom{e+r}{r}-1$ due to Lemma \ref{24}. If $\deg(X)=\mu_S(R)=\binom{e+r}{r}$ then $I_X$ has an $(r+1)$-linear resolution by Theorem \ref{31}. This is impossible because $(I_X)_r\not=0$. Therefore
$$\deg(X)=\mu_S(R)=\binom{e+r}{r}-1.$$
\medskip

\noindent $(b)\Rightarrow (c)$: Suppose that $R$ is Cohen-Macaulay. Due to  Bermejo-Gimenez (see Proposition \ref{25}(c)), the initial ideal $\ini(I_X)$ is minimally generated by a set $\mathcal B$ of monomials in $x_0, \ldots, x_{e-1}$. Moreover, the generating set $\mathcal B$ contains a monomial $m$ of degree $r$ since $\dim_k(I_X)_r=1$.

On the other hand, as we have seen in the proof of $(b)\Rightarrow (a)$, the $S$-module $R$ is generated by all monomials in $x_0, \ldots, x_{e-1}$ of degrees at most $r$ except the monomial $m$. As these are all the standard monomials with respect to $\ini(I_X)+(x_e, \ldots, x_d)$, the set $\mathcal B$ contains all monomials in $x_0, \ldots, x_{e-1}$ of degree $r+1$ which are not multiples of $m$. This proves $(c)$.

\medskip

\noindent $(c)\Rightarrow (b)$: Again by Proposition \ref{25}, the $k$-algebra $R$ is Cohen-Macaulay. Hence $\reg(R)=r$ and $(I_X)_{>r}$ has a linear resolution. The fact $\dim_k(I_X)_r=1$ is clear from the assumption.
\end{proof}

In the next we give some examples of arithmetically Cohen-Macaulay projective subschemes of almost maximal degree.

\begin{example}\label{41-1}({\bf Finite points with almost maximal degree})
Let $\Gamma \subset \bP^e$ be a non-degenerate set of finite points with reduction number $r$. In this case, the reduction number $r(\Gamma)=\reg(R_{\Gamma})$ (see Proposition \ref{21}). Therefore, it is $r$-normal and in particular, $\deg(\Gamma)=\binom{e+r}{r}-1$ if and only if there is only one hypersurface of degree $r$ containing $\Gamma $ and $\reg(R_{\Gamma})=r$ (cf. Example \ref{32-1}).
In this case, $\Gamma$ has the following Betti numbers
$$\beta_{i,r}=\binom{e+r}{i+r}\binom{r+i-1}{r}-\binom{e}{i},$$
for $i=1, \ldots, e$ (cf. Theorem \ref{45}).
\end{example}

\begin{example}\label{42}
Interesting examples of arithmetically Cohen-Macaulay closed subschemes of $\bP^{n+e}$ with almost maximal degree consist of Castelnuovo surface,  Castelnuovo $3$-fold and their higher dimensional analogues. Let $X$ be the Castelnuovo surface in $\bP^4$. It is smooth and arithmetically Cohen-Macaulay of degree $5$. The defining ideal of $X$ is generated by a quadric and two cubics. The $S_0$-minimal free resolution of the homogeneous coordinate ring $R$ is
$$0\rightarrow S_0(-4)^2\rightarrow S_0(-3)^2\oplus S_0(-2) \rightarrow S_0\rightarrow 0$$ with the Betti table

\begin{figure}[!htb]
\begin{tabular}{>{\centering}m{2cm}|>{\centering}m{1cm} >{\centering}m{1cm} c}
			&	0&		1&	2\\
			\hline
		0	&	1&		--&	--\\
		1	&	--&		1&	--\\
		2	&	--&		2&	2
\end{tabular}
\end{figure}

Assume that $S_2=k[x_2,x_3, x_4]$ is a Noether normalization of $R$. Then as an $S_2$-module,
$$R\simeq S_2\oplus S_2(-1)^2\oplus S_2(-2)^2.$$
In particular, $X$ has an almost maximal degree $\binom{2+2}{2}-1$ and $r(X)=2$.
Let $C$ be a generic hyperplane section of $X\subset \bP^4$ which is a smooth curve of degree $5$ and genus $2$
in $\bP^3$. Then, $C$ is also arithmetically Cohen-Macaulay with $r(C)=2$ (see \cite[Example 6.4.2]{Hart77}).
    The Betti table of its homogeneous coordinate ring $R_C$ is the same as that of $R$.
\end{example}

If we do not require the subschemes to be reduced and irreducible, examples are found easily. The following example comes from Theorem \ref{41}.

\begin{example}\label{43}
Let $S_0=k[x_0, \ldots, x_{n+e}]$. Let $J=(x_0, \ldots, x_{e-1})$ and $u\in J$ be a monomial of degree $r$. Put $I=(u)+J^{r+1}$. Then $V(I)\subset \bP^{n+e}$ is a subscheme of almost maximal degree.
\end{example}

In the second half of this section, we are going to prove Theorem \ref{45} which give the  explicite Betti tables of all arithmetically Cohen-Macaulay subschemes of $\bP^{n+e}$ of almost maximal degree. We first compute the Betti table of the subscheme in Example \ref{43}.

\begin{lemma}\label{44}
Let $S_0$, $u$, $J$, $I$ be as in Example \ref{43}. The Betti numbers of the ideal $I$ are
$$\beta_{ij}^{S_0}(I)=\begin{cases}
1 &\mbox{ if } i=0, j=r,\\
\binom{e+r}{i+r+1}\binom{r+i}{r}-\binom{e}{i+1} &\mbox{ if } 0\leq i\leq e-1, j=r+1,\\
0 &\mbox{ otherwise.}
\end{cases}
$$
\end{lemma}
\begin{proof}
We denote $K=J^{r+1}$. Then $(u)\cap K=uJ$. From the short exact sequence
$$0\rightarrow (u)\rightarrow I\rightarrow I/(u)\rightarrow 0,$$
and the fact that $I/(u)\simeq (u)+K/(u)\simeq K/(u)\cap K=K/uJ$ and $(u)\simeq S_0[-r]$, we obtain a short exact sequence
$$0\rightarrow S_0[-r]\rightarrow I\rightarrow K/uJ\rightarrow 0.$$
This gives rise to a long exact sequence
$$\ldots \rightarrow \Tor_{i+1}^{S_0}(k, K/uJ)_{i+j} \rightarrow \Tor_i^{S_0}(k, S_0)_{i+j-r} \rightarrow \Tor_i^{S_0}(k, I)_{i+j} \rightarrow \Tor_i^{S_0}(k, K/uJ)_{i+j} \rightarrow \ldots$$
Since $S_0[-r]$ has the regularity $r$ and $I$ has the regularity $r+1$, the regularity of $K/uJ$ is $r+1$. This together with the fact $K/uJ$ being generated by degree $(r+1)$-elements imply that $K/uJ$ has an $(r+1)$-linear minimal free resolution.

Now we show that the homomorphism $\Tor_{i+1}^{S_0}(k, K/uJ)_{i+j} \rightarrow \Tor_i^{S_0}(k, S_0)_{i+j-r}$ in the long exact sequence above is actually zero. Indeed, if $j\not= r+2$ then $\Tor_{i+1}^{S_0}(k, K/uJ)_{i+j}=0$. If $j=r+2$ then $\Tor_i^{S_0}(k, S_0)_{i+j-r}=0$. This proves the claim. Consequently, we have a short exact sequence
$$0 \rightarrow \Tor_i^{S_0}(k, S_0)_{i+j-r} \rightarrow \Tor_i^{S_0}(k, I)_{i+j} \rightarrow \Tor_i^{S_0}(k, K/uJ)_{i+j} \rightarrow 0,$$
for all $i, j$. This particularly implies that
$$\beta_{ij}^{S_0}(I)=\beta_{i, j-r}^{S_0}(S_0)+\beta_{ij}^{S_0}(K/uJ).$$

In order to compute the Betti numbers of $K/uJ$, we use the short exact sequence
$$0\rightarrow J[-r]\stackrel{*u}{\longrightarrow} K\rightarrow K/uJ\rightarrow 0.$$
Again, there is a long exact sequence
$$\ldots \rightarrow \Tor_{i+1}^{S_0}(k, K/uJ)_{i+j} \rightarrow \Tor_i^{S_0}(k, J)_{i+j-r} \rightarrow \Tor_i^{S_0}(k, K)_{i+j} \rightarrow \Tor_i^{S_0}(k, K/uJ)_{i+j} \rightarrow \ldots$$
Note that all $J[-r]$, $K$ and $K/uJ$ have $(r+1)$-linear resolutions. By analogous argument as in the first part of the proof, we obtain a short exact sequence
$$0 \rightarrow \Tor_i^{S_0}(k, J)_{i+j-r} \rightarrow \Tor_i^{S_0}(k, K)_{i+j} \rightarrow \Tor_i^{S_0}(k, K/uJ)_{i+j} \rightarrow 0,$$
for all $i, j$. Therefore
$$\beta_{ij}^{S_0}(I)=\beta_{i, j-r}^{S_0}(S_0)+\beta_{ij}^{S_0}(K)-\beta_{i,j-r}^{S_0}(J).$$

Now, using Corollary \ref{32}, finally we get an explicit formula for the Betti numbers of $I$, namely,
$$\beta_{ij}^{S_0}(I)=\begin{cases}
1 &\mbox{ if } i=0, j=r,\\
\binom{e+r}{i+r+1}\binom{r+i}{r}-\binom{e}{i+1} &\mbox{ if } 0\leq i\leq e-1, j=r+1,\\
0 &\mbox{ otherwise.}
\end{cases}
$$
\end{proof}

Lemma \ref{44} does not only give the explicit Betti table for a particular subscheme of almost maximal degree but is also very useful when we compute the Betti table for the general case. Now we use this lemma to prove Theorem \ref{45}.

\begin{proof}[\bf Proof of Theorem \ref{45}]

Let $R=S_0/I_X$ be the homogeneous coordinate ring of $X$. We assume that $S=k[x_e, \ldots, x_{n+e}]\subset R$ is a Noether normalization of $R$ with the reduction number $r_S(R)=r$. We consider $R$ as an $S_0$-module and as an $S$-module.

As an $S$-module, the Betti number of $R$ is
$$\beta_{ij}^S(R)=\begin{cases}
\binom{e+j-1}{j} &\mbox{ if } i=0, 0\leq j\leq r-1,\\
\binom{e+r-1}{r}-1 &\mbox{ if } i=0, j=r,\\
0 &\mbox{ otherwise.}
\end{cases}$$
Thus
$$\chi_m^S(R)=\sum_{j=0}^m(-1)^j\beta_{m-j, j}^S(R)=(-1)^m\beta_{0,m}^S(R)=
\begin{cases}
(-1)^m\binom{e+m-1}{m} &\mbox{ if } 0\leq m\leq r-1,\\
(-1)^m\binom{e+r-1}{r}-1 &\mbox{ if } m=r,\\
0 &\mbox{ otherwise.}
\end{cases}$$
Using Corollary \ref{27}(a) which relates the Betti numbers of $R$ over $S$ with those over $S_0$, we are able to find a nice relation between Betti numbers of $R$ over $S_0$, namely,
\[\begin{aligned}
\chi_{m}^{S_0}(R)
&=\sum_{j=0}^e\binom{e}{j}\chi_{m-j}^{S_e}(R)\\
&=\sum_{j=m-r}^e(-1)^{m-j}\binom{e}{j}\binom{e+m-j-1}{m-j}-(-1)^r\binom{e}{m-r}\\
\end{aligned}\]

On the other hand, by Theorem \ref{41}, the initial ideal $\ini(I_X)$ with respect to the degree reverse lexicographic order is generated by all monomials in $x_0, x_1, \ldots, x_{e-1}$ of degree $r+1$ and a monomial in $x_0, x_1, \ldots, x_{e-1}$ of degree $r$. Lemma \ref{44} applies to $\ini(I_X)$ and we have
$$\beta_{ij}^{S_0}(\ini(I_X))=0,$$
for either $j\not=r, r+1$, or $j=r, i>0$, or $j=r+1, i\geq e$. Now, by comparing the Betti numbers of $I_X$ and its initial ideal (see, for example, \cite[Corollary 1.21]{Gre98}), we get
$$\beta^{S_0}_{ij}(I_X)\leq \beta^{S_0}_{ij}(\ini(I_X)),$$
for any $i, j$. So
$$\beta_{ij}^{S_0}(I_X)=0,$$
for either $j\not=r, r+1$, or $j=r, i>0$, or $j=r+1, i\geq e$. This implies that
$$\chi_{i+r}^{S_0}(R)
=\sum_{j=0}^{i+r}(-1)^j\beta_{r+i-j, j}^{S_0}(R)
=(-1)^r\beta_{i, r}^{S_0}(R).$$

Combining these with Lemma \ref{33}, we therefore obtain
\[\begin{aligned}
\beta_{i, r}^{S_0}(R)
&=\sum_{j=i}^e(-1)^{i-j}\binom{e}{j}\binom{e+i+r-j-1}{i+r-j}-\binom{e}{i}\\
&=\binom{e+r}{i+r}\binom{r+i-1}{i-1}-\binom{e}{i}.
\end{aligned}\]
\end{proof}


\section{The non-arithmetically Cohen-Macaulay varieties with almost maximal degree}

In this section we consider projective subschemes with almost maximal degree which are not arithmetically Cohen-Macaulay. It is much more complicated to explore this class of subschemes than the arithmetically Cohen-Macaulay subschemes in the previous section. Fortunately, in the case of reduced and irreducible projective subschemes, we have the following key theorem.

\begin{theorem}\label{51}
Let $X\subset \bP^{n+e}$ be a non-degenerate projective variety of dimension $n$, codimension $e$ and reduction number $r$. Assume that $\deg(X)=\binom{e+r}{r}-1$. Then the homogeneous coordinate ring of $X$ has depth $\geq n$.
\end{theorem}
\begin{proof}
Let $I\subset S_0=k[x_0, \ldots, x_{n+e}]$ be the defining prime ideal of $X$ and $R=S_0/I$. Assume that $S=S_e=k[x_e, \ldots, x_{n+e}]$ is a Noether normalization of $R$ such that the reduction number $r_S(R)=r$. It is enough to prove the assertion for the case $X$ being not arithmetically Cohen-Macaulay.

The proof consists of several steps. Note that we fix the degree reverse lexicographic order on the monomials of $S_0$.

\medskip
\noindent Step 1. We show that the initial ideal $\ini(I)$ has a minimal set of generators consisting of all monomials in $x_0, \ldots, x_{e-1}$ of degree $r+1$ and some monomials  $uv_1, \ldots, uv_s$, where $u$ is a monomial in $x_0, \ldots, x_{e-1}$ of degree $r$ and $v_1, \ldots, v_s$ are monomials in $x_e, \ldots, x_{n+e}$ of positive degree.

Denote the set of all monomials in $x_0, \ldots, x_{e-1}$ of degree $d$ by $T_d$, for $d\geq 0$.

Using Propositions \ref{23} and \ref{25}, we have $\deg(X)\leq \mu_S(R)\leq \binom{r+e}{r}$. Since $X$ is not arithmetically Cohen-Macaulay, the inequality $\deg(X)<\mu_S(R)$ is strict and thus $\mu_S(R)=\binom{e+r}{r}$. It together with Proposition \ref{23} and Lemma \ref{24} shows that the monomials in $T_0\cup T_1\cup\ldots \cup T_r$ are all monomials not being contained in the ideal $\ini(I)+(x_e, \ldots, x_{n+e})$. This also shows that $\ini(I)$ has a minimal set of generators consisting of $T_{r+1}$ and some monomials $u_1v_1, \ldots, u_sv_s$ where $u_i\in T_j$ for some $j\leq r$ and $v_i\in S_+$. Here it is worth noting that $\deg(u_i)>0$ since $S_0/I$ and $S_0/\ini(I)$ have the same Krull dimension.

We proceed to show that $u_1=\ldots=u_s$ by computing the degree of $S_0/\ini(I)$. It suffices to look at the minimal prime ideals of $\ini(I)$. Since
$$\ini(I)=(u_1, u_2v_2, \ldots, u_sv_s, T_{r+1})\cap (v_1, u_2v_2, \ldots, u_sv_s, T_{r+1}),$$
the degree of $S_0/\ini(I)$ is the same as the degree of $S_0/(u_1, u_2v_2, \ldots, u_sv_s, T_{r+1})$. Taking a similar decomposition with respect to $u_2v_2, \ldots, u_sv_s$, we get
$$\deg(S_0/\ini(I))=\deg S_0/(u_1, \ldots, u_s, T_{r+1}).$$
Note that $(u_1, \ldots, u_s, T_{r+1})$ is a monomial ideal and $S_0/(u_1, \ldots, u_s, T_{r+1})$ is a Cohen-Macaulay algebra with a maximal regular sequence $x_e, \ldots, x_{n+e}$. Then, in order to compute the degree of $S_0/(u_1, \ldots, u_s, T_{r+1})$, it suffices to count the monomials in $T_0\cup T_1\cup\ldots\cup T_r$ which are not in $(u_1, \ldots, u_s, T_{r+1})$. On the other hand,
we have totally $\binom{e+r}{r}$ monomials in the set $T_0\cup T_1\cup\ldots\cup T_r$ and
$$\deg(S_0/I)=\deg (S_0/\ini(I))=\deg(S_0/(u_1, u_2, \ldots, u_s, T_{r+1}))=\binom{e+r}{r}-1.$$
So there is only one  monomial of degree at most $r$ in $(u_1, u_2, \ldots, u_s, T_{r+1})$
which is not in the ideal $\ini(I)$. This implies that $u_1=u_2=\ldots=u_s$ which is of degree $r$. Denote this monomial by $u$.

\medskip
\noindent Step 2. The aim of this step is to prove that the equivalence classes in $R$ of the monomials in $(T_0\cup T_1\cup \ldots \cup T_r)\setminus \{u\}$ are linearly independent over $S$.
\smallskip

Let's denote the monomials in the set $(T_0\cup T_1\cup \ldots \cup T_r)\setminus \{u\}$ by $u_1, \ldots, u_N$, where $N=\binom{e+r}{r}-1$. Assume that there are $f_1, \ldots, f_N\in S$ which are not all identically zero and satisfy the relation
$$f_1\bar u_1+\ldots+f_N\bar u_N=0.$$
Let $f=f_1u_1+\ldots+f_Nu_N\in I$. We can assume in addition that $f_1, \ldots, f_N$ are homogeneous polynomials such that $f$ is also homogeneous. Obviously the monomials $u_1, \ldots, u_N\in S_0$ are linearly independent over $S$, so $f\not=0$. Let $\ini(f)=\lambda m_1m_2$, where $\lambda\in k$, $\lambda\not=0$, and $m_1\in(T_0\cup T_1\cup \ldots \cup T_r)\setminus \{u\}$, $m_2\in S$. This contradicts to the fact that $\ini(f)$ lies in the ideal $\ini(I)$, the latter is minimally generated over $S$ by $T_{r+1}\cup \{uv_1, \ldots, uv_s\}$. This completes the proof for Step 2.

\medskip
\noindent Step 3. We prove that $\ini(I)$ is minimally generated by $T_{r+1}\cup \{uv_1\}$ using the fact
that $I$ is a prime ideal.
\smallskip

As we have shown in Step 1, the initial ideal $\ini(I)$ is minimally generated by the monomials in $T_{r+1}$ and some monomials $uv_1, \ldots, uv_s$, where $u\in T_r$ and $v_1, \ldots, v_s\in S_+$. For the convenience, we denote the monomials in $T_{r+1}$ by $m_1, \ldots, m_t$, where $t=\binom{e+r}{r+1}$. Let $h_1, \ldots, h_s$ be part of the reduced Gr{\"o}bner basis with $\ini(h_j)=uv_j$. Observe that the polynomials $h_1, \ldots, h_s$ are irreducible as $I$ is a prime ideal.

Since no trailing terms of any polynomial in the Gr{\"o}bner basis lie in the initial ideal $\ini(I)$, we write
$$h_i=uq_i+\sum_{j=1}^N u_jq_{ij},$$
where $q_i, q_{i1}, \ldots, q_{iN}$ are homogeneous polynomials in $S$ and $\ini(q_i)=v_i$. The irreducibility of $h_i$ implies particularly that $q_i, q_{i1}, \ldots, q_{iN}$ have no common factors of positive degree.

Assume that $s\geq 2$. Then
$$q_2h_1-q_1h_2=\sum_{i=1}^Nu_j(q_2q_{1j}-q_1q_{2j})$$
which is in the ideal $I$. By Step 2, the equivalence classes of the monomials $u_1, \ldots, u_N$ in $R$ are linearly independent over $S$. Hence $q_2q_{1j}-q_1q_{2j}=0$ for all $j=1, \ldots, N$. Since $q_i, q_{ij}\in S$ have unique irreducible factorizations and $q_i, q_{i1}, \ldots, q_{iN}$ have no common factors of positive degree for each $i=1, 2$, it implies that $q_1=\lambda q_2$ for some $\lambda\in k$. This is impossible because it implies $v_1=\ini(q_1)=\lambda\ini(q_2)=\lambda v_2$. Therefore $s=1$ and $\ini(I)$ is minimally generated by $T_{r+1}\cup \{uv_1\}$.

\medskip
\noindent Step 4. Finally, we show that $\depth(R)=n$.
\smallskip

By Step 3, the initial ideal $\ini(I)$ is minimally generated by $T_{r+1}\cup\{uv_1\}$. We have a short exact sequence
$$0\rightarrow S_0/(T_{r+1})\stackrel{*uv_1}{\longrightarrow} S_0/(T_{r+1})\rightarrow S_0/\ini(I) \rightarrow 0.$$
Since $S_0/(T_{r+1})$ is Cohen-Macaulay of dimension $n+1$, we obtain
$$\depth(S_0/\ini(I))\geq n.$$

Since $\beta^{S_0}_{ij}(S_0/I)\leq \beta^{S_0}_{ij}(S_0/\ini(I))$ for any $i, j$ (see, for example, \cite[Corollary 1.21]{Gre98}), we obtain in particular $$\mathrm{proj.dim}_{S_0}(S_0/I)\leq \mathrm{proj.dim}_{S_0}(S_0/\ini(I)).$$
The Auslander-Buchsbaum formula then implies that $\depth(S_0/I)\geq \depth(S_0/\ini(I))\geq n$. As $S_0/I$ is not Cohen-Macaulay, $\depth(S_0/I)=n$.
\end{proof}

In Theorem \ref{51}, we really use the fact that $X$ is a projective variety, i.e., a reduced and irreducible projective subscheme. We will see later in Example \ref{59} that this assumption is necessary and can not be omitted.

The following consequence follows from the proof of Theorem \ref{51}.

\begin{corollary}\label{52}
Let $X\subset \bP^{n+e}$ be a non-degenerate closed subscheme of dimension $n$ and reduction number $r$. Let $I\subset S_0$ be the defining ideal of $X$ and $R=S_0/I$. Let $S=S_e=k[x_e, \ldots, x_{n+e}]$ and suppose the natural map $S\hookrightarrow R$ is a Noether normalization of $R$ with the reduction number $r_S(R)=r$. Again, we fix the degree reverse lexicographic order on the monomials of $S_0$. The following are equivalent:
\begin{enumerate}
\item[(a)] $\deg(X)=\binom{e+r}{r}-1$ and $X$ is not arithmetically Cohen-Macaulay;
\item[(b)] The initial ideal $\ini(I)$ is generated by all monomials in $x_0, \ldots, x_{e-1}$ of degree $r+1$ and some monomials  $uv_1, \ldots, uv_s$, where $u$ is a monomial in $x_0, \ldots, x_{e-1}$ of degree $r$ and $v_1, \ldots, v_s\in S_+$;
\end{enumerate}
If $X$ is a projective variety then the above equivalent statements are equivalent to one of the following statements:
\begin{enumerate}
\item[(c)] $\ini(I)$ is generated by all monomials in $x_0, \ldots, x_{e-1}$ of degree $r+1$ and a monomial  $uv$, where $u$ is a monomial in $x_0, \ldots, x_{e-1}$ of degree $r$ and $v\in S_+$;
\item[(d)] $R$, as an $S$-module, has the Betti numbers
$$\beta_i^S(R)=\begin{cases}
\binom{e+r}{r} &\mbox{ if } i=0,\\
1 &\mbox{ if } i=1,\\
0 &\mbox{ if } i>1;
\end{cases}$$
\item[(e)] $R$, as an $S$-module, has the graded Betti numbers
$$\beta_{i,j}^S(R)=\begin{cases}
\binom{e+j-1}{j} &\mbox{ if } i=0, 0\leq j\leq r,\\
1 &\mbox{ if } i=1, j=\reg(R),\\
0 &\mbox{ if } i=1, j\not=\reg(R) \text{ or } i>1.
\end{cases}$$
\end{enumerate}
\end{corollary}
\begin{proof}
We have shown in Steps 1 and 2 of the proof of Theorem \ref{51} that $(a)\Rightarrow (b)$. For $(b)\Rightarrow (a)$, suppose $\ini(I)=(x_0, \ldots, x_{e-1})^{r+1}+(uv_1, \ldots, uv_s)$ as in $(b)$, then  $R$ is not Cohen-Macaulay due to Proposition \ref{25}. Furthermore, we have
$$\deg(R)=\deg(S_0/\ini(I))=\deg(S_0/(x_0, \ldots, x_{e-1})^{r+1}+(u))=\binom{e+r}{r}-1.$$
So $(a)$ is equivalent to $(b)$. 

Now suppose that $X$ is a projective variety. By Step 3 in the proof of Theorem \ref{51}, we have $(b)\Rightarrow (c)$. The implications $(e)\Rightarrow (d)\Rightarrow (a)$ are obvious. We will show that $(c)\Rightarrow (a)$ and $(a)\Rightarrow (e)$.

\smallskip
\noindent $(c)\Rightarrow (a)$: We denote by $T_d$ the set of the monomials in $x_0, \ldots, x_{e-1}$ of degree $d$.

Assume that $\ini(I)$ has a minimal set of generators consisting of $T_{r+1}$ and some monomials $m_i=uv_i$, $i=1, \ldots, s$, with $u\in T_r$ and $v_i\in S_+$. We have
$$\deg(R)=\deg(S_0/\ini(I))=\deg(S_0/(T_{r+1}, u))=\binom{e+r}{r}-1.$$
On the other hand, since $T_{r+1}$ is a part of a minimal set of generators of $\ini(I)$, the set of standard monomials with respect to $\ini(I)+(x_e, \ldots, x_{n+e})$ is exactly $T_0\cup T_1\cup \ldots \cup T_r$. There are $\binom{e+r}{r}$ such monomials. Using again Proposition \ref{23} and Lemma \ref{24}, we conclude that
$$\mu_S(R)=\binom{e+r}{r}>\deg(R),$$
and $R$ is not Cohen-Macaulay.

\smallskip
\noindent $(a)\Rightarrow (e)$: Suppose $X$ is a projective variety of almost maximal degree, i.e., $\deg(X)=\binom{e+r}{r}-1$, and $X$ is not arithmetically Cohen-Macaulay. We have shown in the proof of Theorem \ref{51} that $R$ is minimally generated over $S$ by all monomials in $x_0, \ldots, x_{e-1}$ of degree from $0$ to $r$. Consequently, we have
$$\sum_{j=0}^r\beta_{0,j}^S(R)=\binom{e+r}{r}.$$
On the other hand, $\beta_{0,j}^S(R)$ is bounded above by the number of monomials in $x_0, \ldots, x_{e-1}$ of degree $j$, i.e.,
$$\beta_{0,j}^S(R)\leq \binom{e+j-1}{j}.$$
Then
$$\sum_{j=0}^r\beta_{0,j}^S(R)\leq \sum_{j=0}^r\binom{e+j-1}{j}=\binom{e+r}{r}.$$
This implies that
$$\beta_{0,j}^S(R)=\binom{e+j-1}{j},$$
for all $j=0, 1, \ldots, r$.

Moreover, we know by Theorem \ref{51} that $R$ has depth $n$. The Auslander-Buchsbaum formula then gives us $\mathrm{proj.dim}_{S_0}(R)=\depth(S_0)-\depth(R)=e+1$. Hence $\mathrm{proj.dim}_S(R)=1$ by Corollary \ref{27}(c). Equivalently, $\beta_{i,j}^S(R)=0$ for all $i>1, j\geq 0$, and there is a positive integer $d>0$ such that
$$\beta_{1,j}^S(R)=\begin{cases}
\mu_S(R)-\deg(R)=1&\mbox{ if } j=d,\\
0&\mbox{ otherwise.}
\end{cases}$$
The minimal free $S$-resolution of $R$ thus is
$$0\rightarrow S(-d-1)\rightarrow \bigoplus_{i=0}^rS(-i)^{\binom{e+i-1}{i}}\rightarrow R\rightarrow 0.$$
It remains to show that $d=\reg(R)$, or  equivalently, $d\geq r$. In the proof of Theorem \ref{51}, we have seen that the initial ideal $\ini(I)$ with respect to the degree reverse lexicographic order is minimally generated by all monomials in $x_0, \ldots, x_{e-1}$ of degree $r+1$ and a monomial $uv$ where $v\in S_+$ and $u$ is a monomial in $x_0, \ldots, x_{e-1}$ of degree $r$ (see Step 3 in the proof of Theorem \ref{51}). Moreover, the equivalence classes in $R$ of the monomials in $x_0, \ldots, x_{e-1}$ of degree at most $r$ except $u$ are linearly independent over the ring $S$. Therefore, $d+1\geq \deg(uv)\geq r+1$, or $d\geq r$.
\end{proof}

So over the Noether normalization $S=S_e$ of the homogeneous coordinate ring $R$, the Betti table of a projective variety $X$ with almost maximal degree is described precisely. It depends on whether the variety is arithmetically Cohen-Macaulay (ACM) or non-arithmetically Cohen-Macaulay (non-ACM) and is either

\begin{figure}[!htb]
\captionsetup[subfigure]{width=4cm}
\subfloat[ACM]{\begin{tabular}{>{\centering}m{2cm}| c}
	&	$0$\\
\hline
$0$	&	$1$\\
$1$	&	$\binom{e}{1}$\\
$2$	&	$\binom{e+1}{2}$\\
$\vdots$	&	$\vdots$\\
$r$	&	$\binom{e+r-1}{r}$-1\\
\end{tabular}}
\hspace{1cm}
\subfloat[non-ACM, $\reg(R)=r$]{\begin{tabular}{>{\centering}m{1.5cm}| >{\centering}m{1.5cm} c}
	&	$0$	&	$1$\\
\hline
$0$	&	$1$	&	--\\
$1$	&	$\binom{e}{1}$&	 	--		\\
$2$	&	$\binom{e+1}{2}$&	 	--		\\
$\vdots$	&	$\vdots$&	 $\vdots$\\
$r$	&	$\binom{e+r-1}{r}$&	 	$1$\\
\end{tabular}}
\hspace{1cm}
\subfloat[non-ACM, $\reg(R)>r$]{\begin{tabular}{>{\centering}m{1.5cm}| >{\centering}m{1.5cm} c}
	&	$0$	&	$1$\\
\hline
$0$	&	$1$	&	--\\
$1$	&	$\binom{e}{1}$&	 	--		\\
$2$	&	$\binom{e+1}{2}$&	 	--		\\
$\vdots$	&	$\vdots$&	 $\vdots$\\
$r$	&	$\binom{e+r-1}{r}$&	 	--\\
$r+1$	&	--&	 	--\\
$\vdots$	&	$\vdots$&	 $\vdots$\\
$\reg(R)$	&	--&	 	$1$\\
\end{tabular}}
\end{figure}

\noindent From these tables we can recover the Betti table over the ring $S_0$ of the variety $X$. This is stated precisely in Theorem \ref{54}. In the next we will present a proof of this theorem. At first, we need the following lemma which is an analogue of Lemma \ref{44} in the arithmetically Cohen-Macaulay case.

\begin{lemma}\label{53}
Let $S_0=k[x_0, \ldots, x_{n+e}]$ and $J=(x_0, \ldots, x_{e-1})$. Let $u\in J$ be a monomial of degree $r$ and $v$ be a non-constant monomial in $x_e, \ldots, x_{n+e}$. Put $I=(uv)+J^{r+1}$. Then $V(I)\subset \bP^{n+e}$ is a subscheme of almost maximal degree. The Betti numbers of the ideal $I$ are as follows:
\begin{enumerate}[(a)]
\item If $\deg(uv)=r+1$ then
$$\beta_{ij}^{S_0}(I)=\begin{cases}
\binom{e+r}{i+r+1}\binom{r+i}{r}+\binom{e}{i} &\mbox{ if } 0\leq i<e, j=r+1,\\
0 &\mbox{ otherwise.}
\end{cases}
$$
\item If $\deg(uv)>r+1$ then
$$\beta_{ij}^{S_0}(I)=\begin{cases}
\binom{e+r}{i+r+1}\binom{r+i}{r}&\mbox{ if } 0\leq i<e, j=r+1,\\
\binom{e}{i} &\mbox{ if } 0\leq i\leq e, j=\deg(uv),\\
0 &\mbox{ otherwise.}
\end{cases}
$$
\end{enumerate}
\end{lemma}
\begin{proof}
Denote by $G(I)$ the minimal set of monomials generating $I$. Then
$$G(I)=\{uv\}\cup G(J^{r+1}).$$
Since $(uv)\simeq S_0[-\deg(uv)]$, the ideals $(uv)$ and $J^{r+1}$ have linear resolutions. By \cite[Corollary 2.4]{FHT09}, $I=(uv)+J^{r+1}$ is a Betti splitting, that means,
$$\beta_{ij}^{S_0}(I)=\beta_{ij}^{S_0}(uv)+\beta_{ij}^{S_0}(J^{r+1})+\beta_{i-1, j+1}^{S_0}((uv)\cap J^{r+1}),$$
for all $i, j$.

We have $(uv)\cap J^{r+1}=v.uJ\simeq J[-\deg(uv)]$ and $(uv)\simeq S_0[-\deg(uv)]$. Hence
$$\beta_{ij}^{S_0}(I)=\beta_{i,j-\deg(uv)}^{S_0}(S_0)+\beta_{ij}^{S_0}(J^{r+1})+\beta_{i-1, j-\deg(uv)+1}^{S_0}(J).$$
It worth noting that for any $t>0$, the ideal $J^t$ defines a subscheme of maximal degree. Now using Corollary \ref{32}, we easily complete the proof.
\end{proof}

Now we are ready to prove Theorem \ref{54}.

\begin{proof}[\bf Proof of Theorem \ref{54}]
Let $I_X\subset S_0$ be the saturated homogeneous defining ideal of $X$ and $R=S_0/I_X$. Changing the variables if necessary, we may assume that $Q=(x_e, \ldots, x_{n+e})$ is a minimal reduction of the irrelevant ideal $R_+$ with the reduction number $r_Q(R)=r$. In particular, $S=S_e=k[x_e, \ldots, x_{n+e}]\hookrightarrow R$ is a Noether normalization of $R$. The algebra $R$ is particularly a finitely generated module over $S_{e-i}$ for $i=0, 1, \ldots, e$.

By Corollary \ref{52}, the initial ideal $\ini(I)$ with respect to the degree reverse lexicographic order is generated by all monomials in $x_0, \ldots, x_{e-1}$ of degree $r+1$ and a monomial  $uv$, where $u$ is a monomial in $x_0, \ldots, x_{e-1}$ of degree $r$ and $v$ is a non-constant monomial in $x_e, \ldots, x_{n+e}$. From the proof of $(a)\Rightarrow (e)$ of Corollary \ref{52} and the usage of Proposition \ref{27}(d), we obtain
$$\reg(R)=\reg_{S_0}(R)=\reg_{S_t}(R)=\deg(uv)-1.$$

By the Cancellation Principle (see \cite[Corollary 1.21]{Gre98} or \cite[Section 3.3]{HH11}), there is for each $m\geq 0$ a complex of $k$-vector spaces
$$V_{\bullet, m}: 0\rightarrow V_{m,m}\rightarrow V_{m-1, m}\rightarrow \ldots\rightarrow V_{1, m}\rightarrow V_{0, m}\rightarrow 0$$
such that
$$V_{i,m}\simeq \Tor_i^{S_0}(\ini(I_X), k)_m,$$
$$H_i(V_{\bullet, m})\simeq \Tor_i^{S_0}(I_X, k)_m,$$
for all $i\geq 0$. Consequently,
$$\beta^{S_0}_{ij}(I_X)\leq \beta^{S_0}_{ij}(\ini(I_X)),$$
for any $i, j$.

Now using Lemma \ref{53}, we obtain the vanishing of Betti numbers, namely,
$$\beta_{ij}^{S_0}(I_X)=0,$$
for either $j\not=r+1, \reg(R)+1$, or $j=r+1, i\geq e$, or $j=\reg(R)+1, i>e$. Moreover, it also shows that in the complex $V_{\bullet, m}$ above, almost all vector spaces are zero except $V_{i, r+1+i}$ and $V_{i, \reg(R)+1+i}$.

To ease the presentation, we consider two cases according to the difference between the regularity and the reduction number.

\smallskip
\noindent Case 1: $\reg(R)$ is either $r$ or $r+1$.

Recall the notation
$$\chi_m^{S_i}(R)=\sum_{j=0}^m(-1)^j\beta_{m-j, j}^{S_i}(R)=\sum_{j=0}^{\reg(R)}(-1)^j\beta_{m-j, j}^{S_i}(R),$$
for $m\in \bZ$. From  Corollary \ref{52}(e), we have
$$\chi_m^{S_e}(R)=\begin{cases}
(-1)^m\binom{e+m-1}{m} &\mbox{ if } 0\leq m\leq r,\\
(-1)^{\reg(R)} &\mbox{ if } m=\reg(R)+1,\\
0 &\mbox{ otherwise.}
\end{cases}$$
Combining this with Corollary \ref{27}(a), we have
$$\chi_m^{S_0}(R)=\sum_{j=0}^e\binom{e}{j}\chi_{m-j}^{S_e}(R)
=\begin{cases}
\sum_{j=0}^e\binom{e}{j}(-1)^{m-j}\binom{e+m-j-1}{m-j}& \mbox{ if } m\leq \reg(R),\\
\sum_{j=0}^e\binom{e}{j}(-1)^{m-j}\binom{e+m-j-1}{m-j}+(-1)^{\reg(R)}\binom{e}{m-\reg(R)-1}& \mbox{ if } m> \reg(R),\\
\end{cases}$$
Finally, using Lemma \ref{33} we get
$$\chi_m^{S_0}(R)=\begin{cases}
(-1)^r\binom{e+r}{m}\binom{m-1}{r} & \mbox{ if } m\leq \reg(R),\\
(-1)^r\binom{e+r}{m}\binom{m-1}{r}+(-1)^{\reg(R)}\binom{e}{m-\reg(R)-1} & \mbox{ if } \reg(R)< m\leq e+\reg(R)+1,\\
0& \mbox{ if } m>e+\reg(R)+1.
\end{cases}$$

If $\reg(R)=r$ then $\chi_m^{S_0}(R)=(-1)^r\beta_{m-r,r}^{S_0}(R)$. This proves (a).

On the other hand, if $\reg(R)=r+1$ then $\chi_m^{S_0}(R)=(-1)^r\beta_{m-r,r}^{S_0}(R)+(-1)^{r+1}\beta_{m-r-1, r+1}^{S_0}(R)$. This proves (b).

\medskip
\noindent Case 2: $\reg(R)>r+1$.

This is an application of the Cancellation Principle (see \cite[Corollary 1.21]{Gre98} or \cite[Section 3.3]{HH11}). Using the notations in the first part of the proof, we have $V_{i,m}=0$ for all $m, i\geq 0$ such that $m-i\not=r+1, \reg(R)+1$. In particular, $V_{i, r+2+i}=0$ and $V_{i, \reg(R)+i}=0$. Consequently, we obtain $H_i(V_{\bullet, m})\simeq V_{i, m}$ for all $i, m$, which induces by the Cancellation Principle the equality $\beta_{ij}^{S_0}(I_X)=\beta_{ij}^{S_0}(\ini(I_X))$. Therefore the Betti number of the homogeneous coordinate ring of $X$ can be computed using Lemma \ref{53}, namely,
$$\beta_{ij}^{S_0}(R)=
\begin{cases}
1&\mbox{ if } i=j=0,\\
\binom{e+r}{i+r}\binom{i+r-1}{r} &\mbox{ if } j=r, 1\leq i\leq e,\\
\binom{e}{i-1} &\mbox{ if } j=\reg(R), 1\leq i\leq e+1,\\
0&\mbox{ otherwise.}
\end{cases}$$
This proves (c).
\end{proof}

There are various examples of non-arithmetically Cohen-Macaulay projective varieties of almost maximal degree. Belows we give some of them and show that all the cases (a), (b), (c) in Theorem \ref{54} actually occur.

\begin{example}\label{55}
Let $C_1$ be a smooth elliptic curve in $\bP^3$ of degree $5$. Then, $C_1$ is an isomorphic projection from a complete embedding of an elliptic curve $C_0 \hookrightarrow \bP^4$. So, $C_1$ is $m$-normal for all $m\ge 2$, but not linearly normal. Thus, we have $\beta_{1,1}(C_1)=h^0(\mathcal {I}_{C_1}(2))=0, \reg(C_1)=3$, the reduction number $r(C_1)=2$ and $\deg(C_1)= \binom{e+r}{r}-1=5$.
The Betti table of its homogeneous coordinate ring is

\begin{figure}[!htb]
\begin{tabular}{>{\centering}m{2cm}|>{\centering}m{1cm} >{\centering}m{1cm} >{\centering}m{1cm} c}
			&	0&		1&	2&	3\\
			\hline
		0	&	1&		--&	--&	--\\
		1	&	--&		--&	--&  --\\
		2	&	--&		5&	5&  1
\end{tabular}
\end{figure}
This Betti table corresponds to (a) in Theorem \ref{54}.
\end{example}

\begin{example}\label{56}
Let $C$ be a smooth curve of degree $9$ and genus $4$ in $\bP^5$. Then, $C$ is a projectively normal embedding and consider its isomorphic projection $C_1\subset\bP^4$ from a center $p\notin Z_2(X)$ where $Z_2(X)$ is the Jacobian scheme defined by $6$ quadrics containing $C$. Note that $Z_2(X)$ is a hypersurface of degree $6$. Then, $C_1$ is $m$-normal for all $m\ge 2$, but not linearly normal
(see \cite[Theorem 2.7]{AR02}). Thus, we have $\beta_{1,1}(C_1)=h^0(\mathcal {I}_{C_1}(2))=0$, the regularity $\reg(C_1)=3$, the reduction number $r(C_1)=2$ and the degree $\deg(C_1)= \binom{e+r}{r}-1=9$. The following Betti table of its homogeneous coordinate ring is type (a) in Theorem \ref{54}:

\newpage
\begin{figure}[!htb]
\begin{tabular}{>{\centering}m{2cm}|>{\centering}m{1cm}>{\centering}m{1cm} >{\centering}m{1cm} >{\centering}m{1cm} c}
			&	0&		1&	2&	3& 4\\
			\hline
		0	&	1&		--&	--&	 --&--\\
		1	&  --&		--&	--&  --&--\\
		2	&  --&		11&	18&   9& 1
\end{tabular}
\end{figure}
\end{example}

\begin{example}\label{57}
Let $C_2$ be the smooth rational curve in $\bP^3$ defined by $(s, t)
\mapsto (s^5, s^4t+s^3t^2, st^4, t^5)$. Then, the curve $C_2$ is of type (5) in Naito's list in \cite[Theorem 1]{Naito02}, and its Betti table is

\begin{figure}[!htb]
\begin{tabular}{>{\centering}m{2cm}|>{\centering}m{1cm} >{\centering}m{1cm} >{\centering}m{1cm} c}
			&	0&		1&	2&	3\\
			\hline
		0	&	1&		--&	--&	--\\
		1	&	--&		--&	--&  --\\
		2	&	--&		4&	3&  --\\
		3	&	--&		1&	2&  1
\end{tabular}
\end{figure}

In particular, $C_2$ is not arithmetically Cohen-Macaulay with the reduction number $r(C_2)=2$, $\reg(C_2)=4$ and $\deg(C_2)=5=\binom{2+2}{2}-1$. The Betti table of $C_2$ corresponds to (b) in Theorem \ref{54}.
\end{example}

\begin{example}\label{58}
Let $C_3$ be the rational curve in $\bP^4$ defined as the image of the map $\bP^1\rightarrow \bP^4$,
\[\begin{aligned}
(s,t)\mapsto (&s^9,s^8 t-s t^8,s^7 t^2-s^2 t^7,\\
&s^6 t^3+s^5 t^4+s^4 t^5+s^3 t^6+s^8 t+s^7 t^2,t^9)
\end{aligned}\]
Using Macaulay 2, we find that the Betti table of $C_3$ is

\begin{figure}[!htb]
\begin{tabular}{>{\centering}m{2cm}|>{\centering}m{1cm}>{\centering}m{1cm} >{\centering}m{1cm} >{\centering}m{1cm} c}
			&	0&		1&	2&	3 & 4\\
			\hline
		0	&	1&		--&	--&	--& --\\
		1	&	--&		--&	--&  -- &--\\
		2	&	--&	     10& 15&  6 & --\\
		3	&	--&		--&	--& --  &--\\
		4	&	--&		--&	--& --  &--\\
		5	&	--&		1&	3 & 3  & 1
\end{tabular}
\end{figure}

Let $S_0=k[x_0,x_1,x_2,x_3,x_4]$ and $I$ be the defining ideal of $C_3$. The homogeneous coordinate ring $R=S_0/I$ has a Noether normalization $k[x_0,x_4]\rightarrow R$ and by computation, $R/(x_0,x_4)\simeq k[x_1,x_2,x_3]/(x_1,x_2,x_3)^3$. Hence the reduction number of $C_3$ is $r(C_3)=2$. Since $\deg(C_3)=9$, the curve $C_3$ has an almost maximal degree. Moreover, $C_3$ is not arithmetically Cohen-Macaulay, and $\reg(C_3)=6$ and the Betti table of $C_3$ corresponds to (c) in Theorem \ref{54}.
\end{example}

We have proved in Theorem \ref{51} that if a non-degenerate projective variety has almost maximal degree then its coordinate ring has depth at least the dimension of the variety. In the next example, we will see that projective varieties can not be replaced by projective subschemes in this theorem. The idea of the example comes from the proof of Theorem \ref{51}.

\begin{example}\label{59}
Let $S_0=k[x_0, \ldots, x_{n+e}]$ as from the beginning. Let $r>0$ and let $u$ be a monomial in $x_0, \ldots, x_{e-1}$ of degree $r$. As in the proof of Theorem \ref{51}, we denote the set of all monomials in $x_0, \ldots, x_{e-1}$ of degree $r+1$ by $T_{r+1}$. Let $I$ be the monomial ideal generated by $T_{r+1}$ and $ux_e, \ldots, ux_{n+e}$. Clearly
$$I=(T_{r+1}, u)\cap (T_{r+1}, x_e, \ldots, x_{n+e}).$$
Then
$$\deg(S_0/I)=\deg(S_0/(T_{r+1}, u))=\binom{e+r}{r}-1,$$
and $\depth(S_0/I)=0$.
\end{example}


\end{document}